\documentclass[a4paper,12pt]{amsart}

\usepackage[centertags]{amsmath}
\usepackage{amssymb,amscd,amsfonts,latexsym,amsthm}
\usepackage{algorithm}
\usepackage{algpseudocode}
\usepackage{t1enc}
\usepackage{graphicx}

\newcommand {\ZZ}  {{\mathbb Z}}

\newtheorem{theorem}{Theorem}
\newtheorem{proposition}{Proposition}
\newtheorem{lemma}{Lemma}
\newtheorem{cor}{Corollary}[section]
\newtheorem{rem}{Remark}[section]

\begin{document}
\title[On a correlational clustering of integers]{On a correlational clustering of integers}
\author[L. Aszal\'os, L. Hajdu and A. Peth\H{o}]{L\'aszl\'o Aszal\'os, Lajos Hajdu and Attila Peth\H{o}}
\address{Institute of Mathematics\\ University of Debrecen \\
H-4010 Debrecen, P.O. Box 12, HUNGARY}
\email{hajdul@science.unideb.hu}
\address{Department of Computer Science, University of Debrecen,
H-4010 Debrecen, P.O. Box 12, HUNGARY}
\email{Aszalos.Laszlo@inf.unideb.hu}
\email{Petho.Attila@inf.unideb.hu}
\thanks{Research supported in part by the OTKA grants NK104208, NK101680, K100339 and by the T\'AMOP-4.2.2.C-11/1/KONV-2012-0001 project. The project is implemented through the New Hungary Development
Plan, cofinanced by the European Social Fund and the European Regional Development Fund.}
\date{\today}

\begin{abstract} Correlation clustering is a concept of machine learning. The ultimate goal of such a clustering is to find a partition with minimal conflicts. In this paper we investigate a correlation clustering of integers, based upon the greatest common divisor.
\end{abstract}

\maketitle

\section{Introduction}
Correlation clustering is a concept of machine learning. It was introduced in Bansal et al. \cite{bbc}, which gives a good overview of the mathematical background as well. Let $G$ be a complete graph on $n$ vertices and label its edges with $+$ or $-$ depending on whether the endpoints have been deemed to be similar or different. Consider a partition of the vertices. Two edges are in conflict with respect to the partition if they belong to the same class, but are different, or they belong to different classes although they are similar. The ultimate goal of correlation clustering is to find a partition with minimal number of conflicts. The special feature of this clustering is that the number of clusters is not specified.

A typical application of correlation clustering is the classification of unknown topics of (scientific) papers. In this case the papers represent the nodes and two papers are considered to be similar, if one of them cite the other. The classes of an optimal clustering is then interpreted as the topics of the papers.

The number of partitions of $n$ vertices  grows exponentially,
hence there is no chance for exhaustive search at computer experiments. Bansal et al. \cite{bbc} proved that to find an optimal clustering is NP-hard. They also presented and analyzed algorithms for approximate solutions of the problem. The correlation clustering can be treated as an optimization problem: minimize the number of conflicts by moving the elements between clusters. Hence one can use the traditional and new optimization algorithms to find near optimal partition.
Bak\'o  and Aszal\'os \cite{ab} have implemented the traditional methods and invented some new ones.

Both above cited papers assumes nothing on the signing of the edges. It is natural to expect that the rule of the signing affects the optimal clustering. This expectation motivated the  investigation of the present paper. To define a systematic signing it is straightforward to label the vertices by natural numbers and the sign between two vertices depends on a number theoretical property of the label of these vertices. As each pair of integers has greatest common divisor, it is natural to call two vertices $a\not=b$ similar if $\gcd(a,b)>1$ and different otherwise. In the sequel we prefer to use the words 'friend' and 'enemy' instead of 'similar' and 'different', respectively. Furthermore we refer to this graph as the network of integers. Note that the behavior of the $\gcd$ among the first $n$ positive integers has been investigated from many aspects; see e.g. a paper of Nymann, \cite{ny}.

Bak\'o  and Aszal\'os \cite{ab} have made several experiments on the network of integers. They discovered that the classes of the near optimal clustering have regular structure. In the sequel denote $p_i$ the $i$-th prime, i.e., $p_1=2, p_2=3,\dots$. Set
$$
S_{i,n} = \{m \; : \; m\le n,\; p_i|m,\; p_j\nmid n\; (j< i)\}.
$$
In other words, $S_{i,n}$ is the set of integers at most $n$, which are divisible by $p_i$, but coprime to the smaller primes. Aszal\'os and Bak\'o found that
\begin{equation} \label{sejtes}
[2,n] \cap \ZZ = \bigcup_{j=1}^{\infty} S_{j,n}
\end{equation}
is highly likely to be an optimal correlation clustering for $n\le 500$. Notice that $S_{j,m}=\emptyset$ for all large enough $j$, i.e., the union on the right hand side is actually finite.

The aim of this paper is to show that for $n_0=3\cdot5\cdot7\cdot11\cdot13\cdot17\cdot19\cdot23 = 111~546~435$ the decomposition \eqref{sejtes} is not optimal. We prove that the number of conflicts in
\begin{equation} \label{teny}
[2,n_0] \cap \ZZ = (S_{1,n_0}\cup\{n_0\}) \cup (S_{2,n_0}\setminus \{n_0\}) \bigcup_{j=3}^{\infty} S_{j,n}
\end{equation}
is less than in \eqref{sejtes} with $n=n_0$. Unfortunately, we are not able to verify that \eqref{teny} is optimal for $n<n_0$. However, we can prove that the natural greedy algorithm (Algorithm 1), presented in the next section, produces the clustering \eqref{sejtes} for $n< n_0$, but \eqref{teny} for $n= n_0$. Thus our result shads some light on the difficulty to find optimal clustering of large graphs.

Applying Algorithm~1 for the network of integers the results behave regularly until a certain large point, but then the regularity disappears. From our construction it will be clear that $n_0$ is the first, but not at all the last integer, which behaves irregularly. For example the numbers $3n_0, 5n_0, 9n_0, \dots$ are odd and are divisible by three, but adjoining them to $S_{1,n}$ causes less conflicts than adjoining them to $S_{2,n}$. Denote by $S_{i,n}^*$ the class, which contains $p_i$ and is produced by Algorithm~1. We have no idea whether these sets have some structure and what is their asymptotic behavior. For example does the limit
$$
\lim_{n \to \infty} \frac{|S_{1,n}^*|}{n}
$$
exist? Or is $\limsup_{n \to \infty} \frac{|S_{1,n}^*|}{n} = 1$ or is it smaller?

The paper is organized as follows. In the second section we present Algorithm~1 and the main theorem. The third section is devoted to the proof of combinatorial lemmata and in the last section we prove the theorem.

\section{Main result}

To find an optimal correlational clustering of a labeled network is an NP-hard problem. To find an approximation of the optimal solution, it is natural to use greedy algorithms. For the network of integers we use the following strategy. The optimal clustering for $\{2\}$ is itself. Assume that we have a partition of $N:=[2,n-1]\cap \ZZ$, then adjoin $n$ to that class, which causes the least new conflicts. The result is a locally optimal clustering, which is not necessarily optimal globally.

\begin{algorithm}
\caption{Natural greedy algorithm}
\begin{algorithmic}[1]
\Require{an integer $n\ge 2$}
\Ensure{a partition $\mathcal P$ of $N$}
\State ${\mathcal P}\gets \{\{2\}\}$;
\If {$n=2$} \textbf{return} $\mathcal P$ \EndIf
\State $m\gets 3$
\While{$m\le n$}
	\State ${\mathcal P}_M\gets {\mathcal P}\cup \{\{m\}\}$
	\State $M\gets \Call{conflicts}{ {\mathcal{P}}_M,m }$
	\Comment{the number of conflicts with respect to the partition ${\mathcal P}_M$ caused by the pairs $(m,a),\ a<m$}
  	\State $C\gets\mbox{ number of classes in } {\mathcal P}$
  	\State $j\gets 1$
  	\While{$j\le C$}
		\State $O \gets \Call{op}{ j,{\mathcal P}}$
		 \Comment{$OP(j,\mathcal P)$ denotes the $j$-th class in the partition $\mathcal P$.}
  		\State ${\mathcal P}_1\leftarrow {\mathcal P}\setminus \{O\}$
  		\State ${\mathcal P}_1\leftarrow {\mathcal P}_1 \cup \{O \cup \{m\}\}$
  		\State $M_1 \gets \Call{NuPair}{ {\mathcal P}_1, m}$
  		 \Comment{ the number of pairs $(m,a)$ with $a<m$ causing a conflict in the partition ${\mathcal P}_1$}
  		\If{$M_1<M$}
  			\State $M\gets M_1$
  			\State ${\mathcal P}_M \gets {\mathcal P}_1$
		\EndIf
	\EndWhile
\EndWhile
\State \textbf{return} ${\mathcal P}_M$
\end{algorithmic}
\end{algorithm}

Starting with a partition of $\{2,\dots,n-1\}$ this algorithm establishes a partition of $\{2,\dots,n\}$ such that the conflicts caused by $n$ is minimal. The output of Algorithm~1 on the input $n$ is denoted by $G(n)$. It is a collection of disjoint sets, whose union is $[2,n] \cap \ZZ$. It is easy to see that
\begin{eqnarray*}
G(3) &=& \{\{2\},\{3\}\}\\
G(4) &=& \{\{2,4\},\{3\}\}\\
G(5) &=& \{\{2,4\},\{3\},\{5\}\}\\
G(6) &=& \{\{2,4,6\},\{3\},\{5\}\}\\
\dots\\
G(15) &=& \{\{2,4,6,8,10,12,14\},\{3,9,15\},\{5\},\{7\},\{11\},\{13\}\},
\end{eqnarray*}
moreover they are optimal as well.

Our main result is the following

\begin{theorem} \label{Lajos}
If $m<n_0=3\cdot5\cdot7\cdot11\cdot13\cdot17\cdot19\cdot23 = 111~546~435$ then
\begin{equation}\label{egy}
G(m) = \bigcup_{j=1}^{\infty}S_{j,m}
\end{equation}
is true, but
$$
G(n_0) = (S_{1,n_0}\cup\{n_0\}) \cup (S_{2,n_0}\setminus \{n_0\})\bigcup_{j=3}^{\infty}S_{j,n_0}.
$$
\end{theorem}

\section{Auxiliary results}

To prove the main theorem we need some preparation. Throughout this paper the number of elements of a set $A$ will be denoted by $|A|$. In the first lemma we characterize that class of $G(n-1)$ to which Algorithm~1 adjoins $n$.

\begin{lemma}\label{hovategyuk}
Let $n>2$ be an integer. Write $G(n-1)=\{P_1,\dots, P_M\}$ and set $P_0=\emptyset$. For $1\le j\le M$ let
$$
E_{j,n} =\{m\;:\; m\in P_j, \gcd(m,n)=1\}
$$
and
$$
B_{j,n} =\{m\;:\; m\in P_j, \gcd(m,n)>1\}.
$$
Define $E_{0,n}=B_{0,n}=\emptyset$. Let $J$ be the smallest index for which $|B_{j,n}|-|E_{j,n}|$ $(j=0,\dots,M)$ is maximal. Then $G(n)=\{P'_0,\dots,P'_M\}$ such that
\[
P'_j =
\begin{cases}
P_j\cup\{n\},& \mbox{if}\; j=J,\\
P_j,& \mbox{otherwise}.
\end{cases}
\]
\end{lemma}

\begin{proof}
Let $K_j$ denote the number of new conflicts, which arise adjoining $n$ to $P_j$ $(j=0,\dots,M)$. Then
$$
K_j =|E_{j,n}|+\sum_{k=0 \atop k\not=j}^M |B_{k,n}|.
$$
Algorithm~1 adjoins $n$ to that $P_{\hat{J}}$ for which $K_{\hat{J}}$ is minimal and if there are more indices $j$ with minimal $K_j$ then $\hat{J}$ is minimal among them. This means that if $K_k\not=K_{\hat{J}}$ and $m\not=k$ then $K_k-K_{\hat{J}} \ge K_k-K_m$. This is equivalent to
$$
|E_{k,n}| + |B_{\hat{J},n}| - |E_{\hat{J},n}| - |B_{k,n}| \ge |E_{k,n}| + |B_{m,n}| - |E_{m,n}| - |B_{k,n}|
$$
and further to
$$
|B_{\hat{J},n}| - |E_{\hat{J},n}| \ge |B_{m,n}| - |E_{m,n}|.
$$
Thus $|B_{m,n}| - |E_{m,n}|$ $(m=0,\dots,M)$ assumes its maximal value at $m=\hat{J}$ and $\hat{J}$ is minimal among the indices with this property. Hence $J=\hat{J}$ and the lemma is proved.
\end{proof}

\begin{cor}\label{newcor} The following assertions are true.
\begin{itemize}
\item[(1)] If $n$ is even, then $n\in S_{1,n}$.
\item[(2)] If $n$ is a prime, then $\{n\}\in G(n)$.
\item[(3)] If the smallest prime factor of $n$ is $p_i$ and $n\in S_{j,n}$, then $j\leq i$.
\end{itemize}
\end{cor}

\begin{proof}
We start with noting that as we pointed out earlier, the assertions hold for $n\le 10$. That is, we have
$$
G(n-1) = \{S_1,\dots,S_M\},
$$
where, for simplicity, we set $S_j = S_{j,n-1}$ $(j=1,\dots,M)$. Put $S_0=\emptyset$.

(1) If $n$ is even then $B_{1,n}=S_1$, thus $|B_{1,n}| = n/2 -1$. If $2\le j\le M$ then $B_{j,n}\subseteq S_j$, thus $|B_{j,n}|\le [(n-1)/p_j] < n/3$. As $n/2-1 > n/3$ for $n\ge 8$ we have
$$
|B_{1,n}|-|E_{1,n}| > |B_{j,n}|-|E_{j,n}|\ \ (j=2,\dots,M).
$$
Hence Algorithm~1 adjoins $n$ to $S_1$, i.e., to the class of even numbers.

(2) If $n=2$ then (2) holds by Step~1. of Algorithms~1. Let $n$ be an odd prime, $G(n-1)=\{S_{1},\dots,S_{M}\}$ and $S_0=\emptyset$. By Lemma \ref{hovategyuk} Algorithm~1 adjoins $n$ to that $S_J$ for which $|B_{j,n}|-|E_{j,n}|$ $(j=0,\dots,M)$ is maximal.

Since $n$ is a prime, $B_{j,n}=\emptyset$ for all $j=1,\dots,M$. Thus $|B_{j,n}|-|E_{j,n}|<0$ $(j =1,\dots,M)$, but $|B_{0,n}|-|E_{0,n}|=0$. Hence $n$ will be adjoined to the empty set, i.e. it will form alone a class in $G(n)$.

(3) We may assume that $n$ is odd and composite. Let $n=q_1^{\alpha_1}\cdots q_t^{\alpha_t}$, where $q_1<\dots<q_t$ are odd primes and $\alpha_1,\dots,\alpha_t$ are positive integers. We obviously have $\{q_1,\dots,q_t\}\subseteq \{p_1,\dots,p_M\}$.

Assume that $q_1 = p_i$. Then every elements of $S_i$ is divisible by $q_1$. Thus $B_{i,n}=S_i$ and $E_{i,n}= \emptyset$, hence $|B_{i,n}|-|E_{i,n}|= |S_i|$.

If $j>i$ then $|B_{j,n}|-|E_{j,n}|\le |S_j|\le|S_i|$. Thus, by Lemma \ref{hovategyuk}, if $n$ will be adjoined to $S_j$ then $j\le i$.
\end{proof}

The next lemma describes a simple, but useful property of the integer part function.

\begin{lemma} \label{egeszresz}
Let $q_1,\dots,q_t$ be pairwise different odd primes, $\alpha_1,\dots,\alpha_t$ positive integers. Let $u$ be a positive integer coprime to $q_i$ $(i=1\dots,t)$ and $n=q_1^{\alpha_1}\cdots q_t^{\alpha_t}$. If $\{i_1,\dots,i_k\}\subseteq \{1,\dots,t\}$ then
\begin{equation} \label{egyegeszresz}
\left[ \frac{n-1}{u q_{i_1}\cdots q_{i_k}} \right] = \left[\frac{\frac{n}{q_{i_1}\cdots q_{i_k}}-1} {u} \right].
\end{equation}
In particular, if $u=2$ then
\begin{equation} \label{ketegeszresz}
\left[ \frac{n-1}{2 q_{i_1}\cdots q_{i_k}} \right] = \frac{n-q_{i_1}\cdots q_{i_k}}{2 p_{i_1}\cdots q_{i_k}} = \frac{n}{2 q_{i_1}\cdots q_{i_k}} - \frac12.
\end{equation}
\end{lemma}

\begin{proof} We have
$$
\left[ \frac{n-1}{u q_{i_1}\cdots q_{i_k}} \right] = \frac{n-m}{u q_{i_1}\cdots q_{i_k}},
$$
where $m$ is the smallest positive integer such that the fraction on the right hand side is an integer. As $q_{i_j}|n$, we must have $q_{i_j}|m$ $(j=1,\dots,k)$ as well, which implies $q_{i_1}\cdots q_{i_k} | m$. This proves \eqref{egyegeszresz}.

If $u=2$ then $m = q_{i_1}\cdots q_{i_k}$ is the smallest positive integer with the required property, because $n-m$ is even.
\end{proof}

The next lemma gives a good approximation for the size of $S_{i,u}$.

\begin{lemma} \label{Sjelemszam}
Let $u$ be an odd integer. Then we have $|S_{1,u}| = \frac{u-1}{2}$. Further, if $p_i$ is an odd prime, then
$$
\left||S_{i,u}|- \frac{u}{p_i} \prod_{\ell=1}^{i-1} \left(1-\frac1{p_{\ell}}\right)\right| \le 2^{i-2}.
$$
\end{lemma}

\begin{proof}
The first statement is obvious. To prove the second one we start with the identity
\begin{eqnarray*}
S_{i,u} &=& \{m\;:\; m\le u,\; p_i|m,\; p_{\ell}\nmid m\; \mbox{for all} \; 1\le \ell < i\}\\
&=& \{m\;:\; m\le u, p_i|m\} \setminus \bigcup_{\ell=1}^{i-1} \{m\;:\; m\le u, p_i\cdot p_{\ell} | m\}.
\end{eqnarray*}
In the remaining of the proof we assume that the elements of the occurring sets are at most $u$. The sieve formula implies
\begin{eqnarray*}
|S_{i,u}| &=& |\{m\;:\; p_i|m\}|\\&
-&\sum_{\ell=1}^{i-1} (-1)^{\ell-1} \sum_{1\le i_1<\dots < i_{\ell}< i}|\{m\;: \; p_i\cdot p_{i_1}|m\}\cap\dots \cap \{m\;:\; p_i\cdot p_{i_{\ell}}|m\}|\\
&=& |\{m\;:\; p_i|m\}| + \sum_{\ell=1}^{i-1} (-1)^{\ell} \sum_{1\le i_1<\dots < i_{\ell}< i}|\{m\;:\; p_i\cdot p_{i_1}\cdots p_{i_{\ell}}|m\}|\\
&=& \sum_{\ell=0}^{i-1} (-1)^{\ell} \sum_{1\le i_1<\dots < i_{\ell}< i}|\{m\;:\; p_i\cdot p_{i_1}\cdots p_{i_{\ell}}|m\}|.
\end{eqnarray*}
Thus
\begin{equation} \label{haromegeszresz}
|S_{i,u}| = \sum_{\ell=0}^{i-1} (-1)^{\ell} \sum_{1\le i_1<\dots < i_{\ell}< i} \left[\frac{u}{p_i\cdot p_{i_1}\cdots p_{i_{\ell}}}\right].
\end{equation}
Using $x-1< [x]\le x$ we obtain
$$
- \sum_{\ell=0\atop {\ell \text{is even}} }^{i-1} {i-1 \choose \ell}\le |S_{i,u}| - \frac{u}{p_i} \prod_{\ell=0}^{i-1} \left(1-\frac1{p_{\ell}}\right) < \sum_{\ell=0\atop {\ell \text{is odd}} }^{i-1} {i-1 \choose \ell}.
$$
As
$$
\sum_{\ell=0\atop {\ell \text{is even}}}^{i-1} {i-1 \choose \ell} = \sum_{\ell=0\atop {\ell \text{is odd}} }^{i-1} {i-1 \choose \ell} = 2^{i-2},
$$
the lemma is proved.
\end{proof}

In the next lemma we prove an estimation for $|B_{j,n}|-|E_{j,n}|$, where
$$
B_{j,n} = \{m\; : \; m \in S_{j,n-1},\; \gcd(m,n)>1\}
$$
and
$$
E_{j,n} = \{m\; : \; m \in S_{j,n-1},\; \gcd(m,n)=1\}.
$$
The elements of $B_{j,n}$ and $E_{j,n}$ are the friends and enemies of $n$ in $S_{j,n-1}$, respectively.

\begin{lemma} \label{Bjn}
Let $q_1<\dots<q_t$ be odd primes, $\alpha_1,\dots,\alpha_t$ positive integers and $n=q_1^{\alpha_1}\cdots q_t^{\alpha_t}$. Let $j\ge2$ be such that $p_j<q_1$. Then
$$
\left||B_{j,n}|-|E_{j,n}| - \frac{n-1}{p_j} \prod_{\ell=1}^{j-1}\left(1-\frac1{p_{\ell}} \right) \left( 1- 2\prod_{k=1}^t\left(1-\frac1{q_k}\right)\right)\right|  \le 2^{t+j-2}.
$$
\end{lemma}

\begin{proof}
As $|B_{j,n}|+|E_{j,n}| = |S_{j,n-1}|$ we have
\begin{equation} \label{Bjnegy}
|B_{j,n}|-|E_{j,n}| = 2 |B_{j,n}| - |S_{j,n-1}| = |S_{j,n-1}| - 2(|S_{j,n-1}| - |B_{j,n}|).
\end{equation}
For $|S_{j,n-1}|$ we can use the estimations of Lemma \ref{Sjelemszam}, thus we have to deal only with the second summand. As $p_j<q$ for all prime factors $q$ of $n$, we have
$$
B_{j,n} = \bigcup_{\ell=1}^t \{m\;:\; m\in S_{j,n-1}, q_{\ell}|m\}.
$$
Using again the sieve formula we get
\begin{eqnarray*}
|B_{j,n}| &=& \sum_{\ell=1}^t (-1)^{\ell-1} \sum_{1\le j_1<\dots<j_{\ell}\le t} \left|\bigcap_{k=1}^{\ell}\{m\;:\; m\in S_{j,n-1}, q_{j_k}|m\}\right|\\
&=& \sum_{\ell=1}^t (-1)^{\ell-1} \sum_{1\le j_1<\dots<j_{\ell}\le t} |\{m\;:\; m\in S_{j,n-1}, q_{j_1}\cdots q_{j_{\ell}}|m\}|.
\end{eqnarray*}
Set $U_{j,\ell}(q_{j_1},\dots,q_{j_{\ell}}) = U_{j,\ell} = \{m\;:\; m\in S_{j,n-1}, q_{j_1}\cdots q_{j_{\ell}}|m\}.$ Then
$$
U_{j,\ell} = \{m\;:\; m\le n-1,p_j\cdot q_{j_1}\cdots q_{j_{\ell}}|m\}\setminus \bigcup_{i=1}^{j-1}\{m\;:\;  p_i p_j q_{j_1}\cdots q_{j_{\ell}}|m\}.
$$
We can compute the number of elements of $U_{j,\ell}$ by using the sieve formula.
\begin{eqnarray*}
|U_{j,\ell}| &=& \left[\frac{n-1}{p_j\cdot q_{j_1}\cdots q_{j_{\ell}}}\right]- \sum_{i=1}^{j-1} (-1)^{i-1} \sum_{1\le h_1<\dots<h_i< j} \left[\frac{n-1}{p_j\cdot q_{j_1}\cdots q_{j_{\ell}}p_{h_1}\cdots p_{h_i}}\right]\\
&=& \sum_{i=0}^{j-1} (-1)^{i-1} \sum_{1\le h_1<\dots<h_i< j} \left[\frac{n-1}{p_j\cdot q_{j_1}\cdots q_{j_{\ell}}p_{h_1}\cdots p_{h_i}}\right].
\end{eqnarray*}
Combining these formulae we get
$$
|B_{j,n}| = \sum_{\ell=1}^t (-1)^{\ell-1} \sum_{1\le j_1<\dots<j_{\ell}\le t} \sum_{i=0}^{j-1} (-1)^{i} \sum_{1\le h_1<\dots<h_i< j} \left[\frac{n-1}{p_j\cdot q_{j_1}\cdots q_{j_{\ell}}p_{h_1}\cdots p_{h_i}}\right].
$$
The last formula together with \eqref{haromegeszresz} implies
$$
|S_{j,n-1}|-|B_{j,n}| = \sum_{\ell=0}^t (-1)^{\ell} \sum_{1\le j_1<\dots<j_{\ell}\le t} \sum_{i=0}^{j-1} (-1)^{i} \sum_{1\le h_1<\dots<h_i< j} \left[\frac{n-1}{p_j\cdot q_{j_1}\cdots q_{j_{\ell}}p_{h_1}\cdots p_{h_i}}\right].
$$
Changing the order of the summation we get
$$
|S_{j,n-1}|-|B_{j,n}| = \sum_{i=0}^{j-1} (-1)^{i} \sum_{1\le h_1<\dots<h_i< j} \sum_{\ell=0}^t (-1)^{\ell} \sum_{1\le j_1<\dots<j_{\ell}\le t} \left[\frac{n-1}{p_j\cdot q_{j_1}\cdots q_{j_{\ell}}p_{h_1}\cdots p_{h_i}}\right].
$$
With Lemma \ref{egeszresz} we get
$$
-1<\left[\frac{n-1}{p_j\cdot q_{j_1}\cdots q_{j_{\ell}}p_{h_1}\cdots p_{h_i}}\right]- \frac{n-1}{p_j\cdot q_{j_1}\cdots q_{j_{\ell}}p_{h_1}\cdots p_{h_i}}\le 0.
$$
Thus
$$
\left|\sum_{\ell=0}^t (-1)^{\ell} \sum_{1\le j_1<\dots<j_{\ell}\le t} \left[\frac{n-1}{p_j\cdot q_{j_1}\cdots q_{j_{\ell}}p_{h_1}\cdots p_{h_i}}\right]- \frac{n-1}{p_j\cdot p_{h_1}\cdots p_{h_i}}\prod_{k=1}^t\left(1-\frac1{q_k}\right)\right| \le 2^{t-1}
$$
and
$$
|S_{j,n-1}|-|B_{j,n}| = \frac{n-1}{p_j}\prod_{k=1}^t\left(1-\frac1{q_k}\right) \prod_{\ell=1}^{j-1}\left(1-\frac1{p_{\ell}}\right) + C,
$$
where the constant $C$ is at most $2^{t+j-3}$. This estimate together with Lemma \ref{Sjelemszam} and \eqref{Bjnegy} gives
$$
|B_{j,n}| - |E_{j,n}|= \frac{n-1}{p_j} \prod_{\ell=1}^{j-1}\left(1-\frac1{p_{\ell}} \right)- 2 \frac{n-1}{p_j}\prod_{k=1}^t\left(1-\frac1{q_k}\right) \prod_{\ell=1}^{j-1}\left(1-\frac1{p_{\ell}}\right) + C_1,
$$
where $|C_1| \le 2^{t+j-3} + 2^{j-2}< 2^{t+j-2}$.
\end{proof}

The next lemma plays a key role in the proof of Theorem \ref{Lajos}. In contrast to the classes of odd numbers, it is possible to give the exact values of the difference of the number of friends and enemies in the class of even numbers.

\begin{lemma} \label{fontos}
Let $n= q_1^{\alpha_1}\cdots q_t^{\alpha_t}$ with $q_1<\dots< q_t$ odd primes and $\alpha_1,\dots, \alpha_t$ positive integers. Then
\begin{eqnarray*}
|E_{1,n}| &=& \frac{\varphi(n)}{2} = \frac{n}2 \left(1-\frac1{q_1}\right)\cdots \left(1-\frac1{q_t}\right),\\
|B_{1,n}| &=& \frac{n-1}2 - |E_{1,n}|.
\end{eqnarray*}
\end{lemma}

\begin{proof}
The statement could be proved by repeating the proof of Lemma \ref{Bjn} and using Lemma \ref{egeszresz}. However, there is a much more direct and simple way, which we present.

Let $A_1=\{h\;:\; 1\le h\le \frac{n-1}2, \gcd(h,n)=1\}$ and $A_2=\{h\;:\; \frac{n+1}2 \le h< n, \gcd(h,n)=1\}$. Then $A_1$ and $A_2$ are disjoint and their union is $A=\{h\;:\; 1\le h < n, \gcd(h,n)=1\}$. Plainly $|A| = \varphi(n)$. The mapping $\psi\;:\; h\mapsto n-h$ is bijective between $A_1$ and $A_2$. Moreover, $\psi(h)$ is odd if and only if $h$ is even. Thus the number of even positive integers, which are coprime to $n$ is $\varphi(n)/2$. As $E_{1,n}$ is exactly the set of even integers, less than and coprime to $n$, the proof is complete.
\end{proof}

\section{Proof of Theorem \ref{Lajos}}

Despite of the lengthy preparation, the proof of Theorem \ref{Lajos} is complicated. First we confirm the cases where $n$ is odd, $3\mid n$, and prove the second assertion. The hard part is to prove that \eqref{egy} is true for $n<n_0$. This is done by a combination of comparison of the estimations of Lemmata \ref{Sjelemszam} and \ref{Bjn}, some computer search and finally application of a tool from prime number theory.

\subsection{The cases $n$ is odd with $3\mid n$, and proving the second assertion}
Let $n> 2$ be an integer and assume that \eqref{egy} holds for all $m< n$. Suppose further that $k=2$, i.e. the smallest prime factor of $n$ is $q_1=3$. Then by Lemma \ref{fontos} we have $|B_{1,n}|-|E_{1,n}|=\frac{n-1}2 - \varphi(n)$. A simple computation shows that $|S_2| = \frac{n-3}6$, which is a bit stronger then the statement of Lemma \ref{Sjelemszam}.

By Lemma \ref{hovategyuk}, $n$ is adjoined to $S_1$ precisely when $|S_2|< |B_{1,n}|-|E_{1,n}|$. This inequality implies $\frac{n-3}6 < \frac{n-1}2 - \varphi(n)$, which is equivalent to $\varphi(n)<\frac{n}3$. Using the explicit form of $\varphi(n)$ and dividing by $n$ we get $\left(1-\frac13\right)\cdot \left(1-\frac1{q_2}\right)\cdots \left(1-\frac1{q_t}\right)<\frac13$. Thus Algorithm~1 adjoins $n$ to $S_1$ if and only if
\begin{equation} \label{ketto}
\left(1-\frac1{q_2}\right)\cdots \left(1-\frac1{q_t}\right)<\frac12.
\end{equation}
Inequality \eqref{ketto} is independent from the exponents $\alpha_1,\dots,\alpha_t$, thus $n$ is minimal with this property if all exponents are one. For fixed $t$ the number $n$ is minimal if $q_1,\dots,q_t$ are consecutive odd primes starting with $q_1=3$. A simple computation shows that the smallest $n$ which is divisible by $3$ and satisfies \eqref{ketto} is $n_0$.

\begin{rem} An analogous computation shows that $n_1=5\cdot p_4\cdots p_{14} = 2~180~460~221~945~005$ is a candidate to be the smallest odd integer, which is not divisible by $3$ and is adjoined to $S_1$. However, as $n_1$ is much larger than $n_0$ and many odd integers between $n_0$ and $n_1$, e.g. $3n_0, 5n_0, 9n_0,\dots$ are adjoined to $S_1$ we are not sure that for example $n_1'=5\cdot p_4\cdots p_{13}$ will belong to $S_1$ or to $S_3$.
\end{rem}

\subsection{Numbers with middle sized factors}

Let $n= q_1^{\alpha_1}\cdots q_t^{\alpha_t}<n_0$ be such that $p_i=q_1<q_2\dots<q_t$. For $i\le 2$ Theorem \ref{Lajos} is already proved. Assume that $i\ge 3$. By Lemma~\ref{hovategyuk} Algorithm~1 adjoins $n$ to $S_{i,n-1}$ if and only if
\begin{equation} \label{feltetel}
|B_{j,n}| - |E_{j,n}| < |S_{i,n-1}|
\end{equation}
holds for all $1\le j<i$. By Lemmata \ref{Sjelemszam}, \ref{Bjn} we have
$$
|S_{i,n-1}| \ge \frac{n-1}{p_i} \prod_{\ell=1}^{i-1} \left(1-\frac1{p_{\ell}}\right) - 2^{i-2}
$$
and
$$
|B_{j,n}| - |E_{j,n}| \le \frac{n-1}{p_j} \prod_{\ell=1}^{j-1}\left(1-\frac1{p_{\ell}} \right) \left(1- 2\prod_{k=1}^t\left(1-\frac1{q_k}\right)\right)+  2^{t+j-2}.
$$
Thus if
$$
\frac{n-1}{p_j} \prod_{\ell=1}^{j-1}\left(1-\frac1{p_{\ell}} \right) \left(1- 2\prod_{k=1}^t\left(1-\frac1{q_k}\right)\right)+  2^{t+j-2} \le \frac{n-1}{p_i} \prod_{\ell=1}^{i-1} \left(1-\frac1{p_{\ell}}\right) - 2^{i-2},
$$
then \eqref{feltetel} holds, too. The last inequality is equivalent to
$$
\frac{n-1}{p_i} \prod_{\ell=1}^{i-1} \left(1-\frac1{p_{\ell}}\right) + \frac{n-1}{p_j} \prod_{\ell=1}^{j-1}\left(1-\frac1{p_{\ell}} \right) \left( 2\prod_{k=1}^t\left(1-\frac1{q_k}\right)-1 \right) \ge 2^{t+j-2} + 2^{i-2}.
$$
For fixed $t$ and $i$ the product $\prod_{k=1}^t\left(1-\frac1{q_k}\right)$ assumes its smallest value if the $q_k$-s are the $t$ consecutive primes starting with $p_i$. Thus if $n_1=n_1(i,j,t)$ denotes the smallest $n$, which satisfies
$$
n \ge 1+ (2^{t+j-2} + 2^{i-2})/T,
$$
where
$$
T = \frac{1}{p_i} \prod_{\ell=1}^{i-1} \left(1-\frac1{p_{\ell}}\right) + \frac{1}{p_j} \prod_{\ell=1}^{j-1}\left(1-\frac1{p_{\ell}} \right) \left( 2\prod_{k=1}^t\left(1-\frac1{q_k}\right)-1 \right)
$$
then \eqref{feltetel} holds for all $n\ge n_1$ having exactly $t$ prime factors, from which the smallest is $p_i$.

We computed $n_1(i,j,t)$ for all triplets $(i,j,t)$ with $3\le i \le 20, 3\le j\le i-1, 1\le t\le t_i$, where $t_i$ is the largest $t$ such that $\prod_{k=0}^{t-1} p_{i+k} \le n_0$. Remark that $n_1(i,j,t)>n_0$ for $i\ge 20$. By our computation $n_1(i,j,t)$ is a monotone increasing function of $j$ for fixed $(i,t)$, thus we displayed in Table \ref{T2} only the values $n_1(i,i-1,t)$.

\begin{table}
\begin{tabular}{|c|c|c|c|c|}
\hline
$i=3, p=5$&  $i=4, p=7$& $i=5,p=11$&  $i=6,p=13$& $i=7,p=17$  \\
\hline
$t, n1$ & $t, n1$ & $t, n1$ & $t, n1$ & $t, n1$ \\
\hline
$1,24$ &$1,93$     & $1,308$ &$1,953$& $1,2521$ \\ \hline
$2, 46$ & $2, 159$ & $2,514$ & $2, 1533$ & $2, 4033$  \\ \hline
$3, {\it 92}$ &$3, {\it 297}$  & $3, {\it 933}$ &$3, {\it 2720}$  & $3, {\it 7091}$   \\ \hline
$4, {\it 196}$ &$4, {\it 583}$ & $4, {\it 1813}$ &$4, {\it 5157}$ & $4, {\it 13318}$   \\ \hline
$5, {\it 422}$ &$5, {\it 1194}$ & $5, {\it 3654}$ &$5, {\it 10151}$ & $5, {\it 26186}$  \\ \hline
$6, {\it 929}$ &$6, {\it 2480}$ & $6, {\it 7471}$ &$6, {\it 20484}$ &   \\ \hline
$7$, {\it 2044} && & &    \\ \hline \hline
$i=8,p=19$ &  $i=9, p=23$& $i=10,p=29$&  $i=11,p=31$& $i=12,p=37$\\
\hline
$t, n1$ & $t, n1$ & $t, n1$ & $t, n1$ & $t, n1$ \\
\hline
$1,6531$ &$1,15889$     & $1,40751$ &$1,98726$& $1,228806$ \\ \hline
$2, 10285$ & $2, 24799$ & $2,63466$ & $2, 152425$ & $2, 352739$ \\ \hline
$3, 17815$ &$3, 42901$  & $3, 109166$ &$3, 260747$  & $3, 603486$   \\ \hline
$4, {\it 33246}$ &$4, {\it 79673}$ & $4, {\it 202185}$ &$4, {\it 481157}$ & $4, {\it 1112639}$   \\ \hline
$5, {\it 64728}$ &$5, {\it 154809}$ & $5, {\it 392470}$ &$5, {\it 929762}$ &   \\ \hline
 \hline
$i=13,p=41$ &  $i=14, p=43$& $i=15,p=47$&  $i=16,p=53$& $i=17,p=59$\\
\hline
$t, n1$ & $t, n1$ & $t, n1$ & $t, n1$ & $t, n1$ \\
\hline
$1,542016$ &$1,1198905$     & $1,2623122$ &$1,5937759$& $1,13554766$ \\ \hline
$2, 833706$ & $2, 1838933$ & $2,4014787$ & $2, 9071489$ & $2, 20693167$ \\ \hline
$3, 1421830$ &$3, 3126106$  & $3, 6813569$ &$3, 15389987$  & $3, 35045873$   \\ \hline
$4, {\it 2612025}$ &$4, {\it 5727822}$ & $4, 12481252$ &$4, 28153619$ & $4, 64044517$   \\ \hline
 \hline

$i=18,p=61$ &  $i=19, p=67$& &  & \\
\hline
$t, n1$ & $t, n1$ & &  & \\
\hline
$1,29627101$ &$1,64068095$     & &&  \\ \hline
$2, 45131528$ & $2, 97553073$ & &  &  \\ \hline
$3, 76322166$ &  &  &  &   \\ \hline

\end{tabular}\\[1.5 ex]
\caption{\label{T2}}
\end{table}

If $n_1(i,i-1,t)> \prod_{k=0}^{t-1}p_{i+k}$, which appeares often, then \eqref{feltetel} cannot hold for the pair $(i,t)$. These values are displayed in italics in Table \ref{T2}.

\subsection{Verification of the statement for numbers with one and two prime factors}
The following lemma verifies Theorem \ref{Lajos} if $n$ is a prime power.
\begin{lemma} \label{primhatvany}
Let $p=p_i\ge 3$ be a prime. If $p\le 67$ and $p^{\alpha}<n_0, \alpha>0$ then $p^{\alpha}\in S_{i,n}$. Generally, if $\alpha\le 4$ then $n=p^{\alpha} \in S_{i,n}$.
\end{lemma}

\begin{proof}
If $\alpha=1$ then as for all $j<i$ $S_{i,n-1}=B_{j,n}=\emptyset$ and $|N_{j,n}|>0$ the statement is true. The statement is valid for $p=3$ too because the smallest integer divisible by three, which lands in $S_1$ is $n_0$.

Let $\alpha>1$ and $j<i$. Then we have
\begin{equation}\label{baratok}
B_{j,p^{\alpha}}= \bigcup_{k=1}^{\alpha-1}p^k E_{j,p^{\alpha-k}}.
\end{equation}
If $m\in B_{j,p^{\alpha}}$, then either $m$ is divisible only by the first power of $p$, thus $m/p\in E_{j,p^{\alpha-1}}$ or $m$ is divisible by a higher power of $p$, in which case $m/p\in B_{j,p^{\alpha-1}}$, i.e.
$$
B_{j,p^{\alpha}} = pE_{j,p^{\alpha-1}} \cup p B_{j,p^{\alpha-1}}.
$$
Using this identity we get the proof of \eqref{baratok} by induction.

{\it Case 1, $\alpha=2$.} Then $|B_{j,p^{2}}| = |E_{j,p}|$. By Tchebishev's theorem there exists at least one prime $q$ with $p/p_j< q < p$. Hence $p_jq\in E_{j,p^{2}} \setminus E_{j,p}$, which implies $|E_{j,p^{2}}| > |B_{j,p^{2}}| = |E_{j,p}|$. Thus $p^2 \in S_{i,n}$, in particular $7^2\in S_{4,n}, 11^2\in S_{5,n}$. These facts together with the entries $(i,t)=(3,1), (4,1)$ of Table \ref{T2} show that the assertion is true for $p=5,7$.

{\it Case 2, $\alpha=3$.} Then there exists a prime $q$ with $p^2/p_j< q < p^2$, i.e. $p_jq\notin E_{j,p^{2}}$. If $m\in E_{j,p}$ then $qm \le qp <p^3$, thus
$$
|E_{j,p^{3}}| \ge |B_{j,p^{3}}| = |E_{j,p^{2}}| + |E_{j,p}|,
$$
and this case is proved. By Table \ref{T2} the assertion is true for $p\le 19$.

{\it Case 3, $\alpha=4$.} Identity \eqref{baratok} implies
$$
|B_{j,p^{4}}| = |E_{j,p^{3}}| + |E_{j,p^{2}}| + |E_{j,p}|.
$$
We have plainly $E_{j,p^{2}} = E_{j,p} \cup E_{j,(p,p^2/p_j]} \cup E_{j,(p^2/p_j,p^2]}$, where the second and third sets include all elements of $E_{j,p^{2}}$, which belong to the interval $(p,p^2/p_j]$ and $(p^2/p_j, p^2]$ respectively. As $p_j\ge 5$ there exist at least two different primes $q_1,q_2$ with $p^3/p_j<q_1,q_2<p^3$, i.e.
$$q_kE_{j,p} \cap E_{j,p^{3}} =q_1E_{j,p} \cap q_2E_{j,p} = \emptyset, k=1,2.$$
Further, there exist primes $q_3,q_4$ with $p^2 < q_3 < 2p^2$ and $p^2/2 < q_4 < p^2$. The sets $q_kE_{j,p}, k=1,2, q_3E_{j,(p,p^2/p_j]}$ and $q_4 E_{j,(p^2/p_j,p^2]}$ are by the construction pairwise disjoint. If $m \in q_3E_{j,(p,p^2/p_j]} \cap q_4 E_{j,(p^2/p_j,p^2]}$ then $m$ is divisible by the pairwise different primes $q_3,q_4,p_j$, thus $m\ge p_jq_3q_4> p_j p^4/2>p^4$, which is a contradiction.
\end{proof}

The next lemma verifies Theorem \ref{Lajos} for integers, which have two different prime divisors and the smaller is at most $53$.
\begin{lemma} \label{ketprim}
Let $p=p_i\ge 3$ and $q>p$ be primes. If $p\le 53$ and $p^{\alpha}q^{\beta}<n_0, \alpha,\beta>0$ then $p^{\alpha}q^{\beta}\in S_{i,n}$. Generally, if $q < p^3$ then $n=pq \in S_{i,n}$.
\end{lemma}

\begin{proof}
The idea of the proof is similar to the proof of Lemma \ref{primhatvany}. The assertion is true for $i=3$ because the smallest integer divisible by three and not lying in $S_{2,n}$ is $n_0$. We start the proof of the general case with the identity
\begin{equation} \label{ketprimazonossag}
B_{j,pq} = qE_{j,p} \cup p(E_{j,q} \cup B_{j,q}) = qE_{j,p} \cup pE_{j,q} \cup p^2 E_{j,\lfloor q/p \rfloor} \cup p^2 B_{j,\lfloor q/p^2 \rfloor} .
\end{equation}
If $p^2>q$ then $B_{j,\lfloor q/p^2 \rfloor} = \emptyset$ and we get
$$
B_{j,pq} = qE_{j,p} \cup pE_{j,q} \cup p^2 E_{j,\lfloor q/p \rfloor}.
$$
There exist primes $r_1,r_2$ with $q/2<r_1<q$ and $p^2/4<r_2<p^2, r_2\not=q,r_1$. Then
$$
E_{j,pq} \supset r_1E_{j,p} \cup E_{j,q} \cup r_2 E_{j,\lfloor q/p \rfloor}
$$
and the sets on the right hand side are disjoint. Thus if $q<p^2$ then $pq\in S_{i,pq}$. Combining this with Table \ref{T2}, implies the assertion for $p\le 17$.

If $p^2<q<p^3$ then \eqref{ketprimazonossag} implies
$$
B_{j,pq} = qE_{j,p} \cup pE_{j,q} \cup p^2 E_{j,\lfloor q/p \rfloor} \cup p^3 E_{j,\lfloor q/p^2 \rfloor}.
$$
If $q\le 4p^2$ then we have $E_{j,\lfloor q/p^2 \rfloor}= \emptyset$ and the argument of the last case works here too. The other extreme case $q>p^3/2$ also needs a separate analysis. Hence we assume $4 p^2< q\le p^3/2$.

We have $E_{j,\lfloor q/p \rfloor} = E_{j,p} \cup E_{j,(p,\lfloor q/p \rfloor]}$, where $E_{j,(p,\lfloor q/p \rfloor]}$ denotes the set of those elements of $E_{j,\lfloor q/p \rfloor}$, which belong to the interval $(p,\lfloor q/p \rfloor]$.
There exist primes $r_k, k=1,2,3,4$ such that $r_1,r_2\in (q/4,q), r_1\not= r_2, \;r_3\in (p^2/2,p^2), r_4\in (p^3/16,p^3), r_4\not= q,r_1,r_2$. Plainly $r_3<r_1,r_2,r_4$. Then
$$
E_{j,pq} \supset E_{j,q} \cup r_{1}E_{j,p} \cup r_{2}E_{j,p} \cup r_{3}E_{j,(p,\lfloor q/p \rfloor]} \cup  r_4 E_{j,\lfloor q/p^2 \rfloor}
$$
and the sets on the right hand side are disjoint. The first claim holds because of the choice of the size of the $r_k$'s. If $m\in r_i E_{j,p}, i=1,2$ then $r_ip_j|m$, thus $m\ge r_ip_j> qp_j/4>q$, hence $m \notin E_{j,q}$. A similar argument shows that $E_{j,q} \cap r_4 E_{j,\lfloor q/p^2 \rfloor} = \emptyset$. If $m \in E_{j,(p,\lfloor q/p \rfloor]}$ then $mr_3> p p^2/2 \ge q$, i.e. $E_{j,q} \cap r_3 E_{j,(p,\lfloor q/p \rfloor]} = \emptyset$, which finishes the proof of this case.

There remains to treat the case $q> p^3/2.$ We have also $E_{j,\lfloor q/p \rfloor} = E_{j,2p} \cup E_{j,(2p,\lfloor q/p \rfloor]}$. There exist primes $r_k, k=1,2,3,4$ such that $r_1,r_2\in (q/8,q/2), r_3\in (p^2/2,p^2), r_4\in (p^3/16,p^3), r_4\not= q,r_1,r_2$. Plainly we also have $r_3<r_1,r_2,r_4$. Then
$$
E_{j,pq} \supset E_{j,q} \cup r_{1}E_{j,p} \cup r_{2}E_{j,2p} \cup r_{3}E_{j,(2p,\lfloor q/p \rfloor]} \cup  r_4 E_{j,\lfloor q/p^2 \rfloor}
$$
and the sets on the right hand side are disjoint. The proof is the same except that if $m \in r_{3}E_{j,(2p,\lfloor q/p \rfloor]}$ then $mr_3> 2p p^2/2 >q$. The proof of the lemma is complete.
\end{proof}

\subsection{Numbers with three prime factors}
Up to now we proved Theorem \ref{Lajos} for integers with at most two prime divisors, such that the smaller is at most $53$. Unfortunately we did not found any meaningful generalization of Lemma \ref{ketprim} to integers with three prime divisors. By Table 1 the smallest prime factor of such a candidate is at least $19$. For each prime $29\le p \le 43$ we computed all integers, which are divisible by $p$, lie below the bound given in Table 1 and have three different prime divisors, which are at least $p$. Their number, $n(p)$ is given in Table \ref{T3}.

\begin{table}
\begin{tabular}{|c|c|c|c|c|c|c|c|c|}
\hline
$p$&  19& 23&  29& 31 & 37 & 41 & 43  \\
\hline
$n(p)$ & 4 & 18 & 65 & 216 & 513 & 1302 & 3097 \\
\hline
\end{tabular}\\[1.5 ex]
\caption{\label{T3}}
\end{table}

Fix $p=p_i$. For each candidates $n$ we computed $|B_{i-1,n}|-|E_{i-1,n}|$. For this purpose we used a variant of the wheel algorithm, see e.g. \cite{w}. In our case this listed efficiently the elements of $S_{i-1,n}$ because we know than they are divisible by $p_i$ and, on the other hand, relative prime to $2,3,5,7$. For each $m$ produced by the wheel algorithm we computed the $\gcd(m, \prod_{j=5}^{i-2}j)$. If this is not one, then $m$ does not belong to $S_{i-1,n}$, otherwise we added one to the counter of $|E_{i-1,n}|$ or $|B_{i-1,n}|$ according as $\gcd(m,n)=1$ or not. We found in each case that $|B_{i-1,n}|-|E_{i-1,n}|<0$, which means $n$ cannot be adjoined to $S_{i-1,n}$. The total computational time was some minutes on a notebook.

\begin{rem}The extension of the computation for larger values of $p$ is very time consuming. We performed the same procedure for $p=67$, in which case the number of candidates is $90338$. The total computational time on the same computer took about one day. The application of tools of the prime number theory saved lot of computer time and is much more elegant.
\end{rem}

\subsection{Numbers with large prime factors}
In the previous sections we proved Theorem \ref{Lajos} for integers such that their smallest prime factor is at most $47$.
\begin{proposition} Let $n_0=111546435$. Assume that for any $n$ with $n<n_0$ having a prime factor $\leq 37$, the partition is the expected one, i.e. \eqref{sejtes} holds. Suppose that $n<n_0$, such that $n$ is composed of primes $\geq 41$. Then $n$ belongs to the class $S_{i,n}$, where $p_i$ is the smallest prime divisor of $n$.
\end{proposition}

To prove the Proposition, we need the following lemma.

\begin{lemma}
\label{lemnew1}
Let $x\geq n_0^{2/3}$ and $t\geq 41$ be real numbers. Then we have
$$
\pi(x)-\pi(\sqrt{x})>18\pi(x/t)+56.
$$
\end{lemma}

\begin{proof} By a classical result of Rosser and Schoenfeld \cite{rs} we have
$$
{\frac{x}{\log x}}\left(1+{\frac{1}{2\log x}}\right)<\pi(x)<
{\frac{x}{\log x}}\left(1+{\frac{3}{2\log x}}\right)
$$
for $x\geq 59$. Hence the statement easily follows by a simple calculation.
\end{proof}

\begin{proof}[Proof of the Proposition] First observe that $n$ cannot have more than four prime divisors, counted with multiplicity. Write $n=q_1\cdots q_I$ with $2\leq I\leq 4$, $41\leq q_1\leq \dots \leq q_I$. Take an arbitrary prime $p_\ell$ with $37\leq p_\ell < q_1$. We show that $n$ cannot be put into the class $S_{\ell,n}$. For this it is sufficient to show that $n$ has more enemies in $S_{\ell,n}$ than friends.

First observe that the elements of $S_{\ell,n}$ have at most four prime divisors, counted with multiplicity. Hence the friends of $n$ in this class have the form
$$
p_\ell q_ir_1r_2,\ p_\ell q_iq_jr_1,\ p_\ell^2 q_ir_1,\ p_\ell q_ir_1,
$$
$$
p_\ell q_i,\ p_\ell q_iq_j,\ p_\ell q_iq_jq_k,\ p_\ell^2 q_i,\ p_\ell^2 q_iq_j,\ p_\ell^3q_i
$$
where $r_1$ and $r_2$ are primes $>p_\ell$ distinct from
$q_1,\dots,q_I$, and the indices $i,j,k$ may be equal.

Observe that if $p_\ell q_ir_1r_2$ is a friend of $n$, then
$$
p_\ell r_1r_2,p_\ell^2 r_1r_2,p_\ell^2 r_1,p_\ell^3 r_1\in S_{\ell,n}
$$
are distinct enemies of $n$. We show that the number of other friends of $n$ is smaller than the number of enemies of $n$ in $S_{\ell,n}$ of the form $p_\ell q$ where $q$ is a prime greater than $p_\ell$. For this first note that the remaining friends of $n$ are of the form $qp_\ell P$ or $p_\ell P$, where $P$ is a product of a power of $p_\ell$ and of $q_i$-s, and $q$ is a prime $>p_\ell$.
The number of friends of the form $p_\ell P$ is at most $52$, since $P$ can be
$$
q_i,\ p_\ell q_i,\ p_\ell^2q_i,\ q_i^2,\ p_\ell q_i^2,\ q_iq_j,\ p_\ell q_iq_j,\ q_i^3,\ q_i^2q_j,\ q_iq_j^2,\ q_iq_jq_k
$$
with $1\leq i<j<k\leq 4$, while for the friends of the form $qp_\ell P$ we have $q\leq n_0/p_\ell q_1$, and here $P$ may take at most the $18$ values
$$
q_i,\ \ \ p_\ell q_i,\ \ \ q_i^2,\ \ \ q_iq_j\ \ \ (1\leq i<j\leq 4).
$$
Hence the number of friends of $n$ in $S_{\ell,n}$ of the latter form is at most $18\pi(n_0/p_\ell q_1)$. On the other hand, the number of enemies of $n$ in $S_{\ell,n}$ of the form $qp_\ell$ is at least $\pi(n_0/p_\ell)-\pi(p_\ell)-4$.

Set $x:=n_0/p_\ell q_1$ and $t:=q_1$. If we further suppose that $p_\ell\leq \sqrt[3]{n_0}$, then by Lemma \ref{lemnew1} we have
$$
\pi(n_0/p_\ell)-\pi(p_\ell)\geq \pi(n_0/p_\ell)-\pi(\sqrt{n_0/p_\ell})
>18\pi(n_0/p_\ell q_1)+56
$$
and the statement follows.

So we are left with the case $p_\ell>\sqrt[3]{n_0}$. Then obviously $n=q_1q_2$ with $p_\ell<q_1\leq q_2$. However, then obviously, $n$ has the only friends
$$
p_\ell q_1,\ \ \ p_\ell q_2
$$
in $S_{\ell,n}$. Since $p_\ell$ and $p_\ell^2$ are enemies of $n$ in $S_{\ell,n}$, and all the elements of $S_{i,n}$ with $p_i=q_1$ are friends of $n$, the statement also follows in this case.
\end{proof}

\end{document}